\documentclass[12pt]{article}
\usepackage{amsmath,amsthm,amssymb,mathtools}
\usepackage{booktabs,tabularx,array}
\mathtoolsset{showonlyrefs}
\usepackage[colorlinks,linkcolor=red,anchorcolor=blue,citecolor=magenta]{hyperref}
\usepackage[numbers,sort&compress]{natbib}
\newtheorem{theorem}{Theorem}[section]
\newtheorem{lemma}[theorem]{Lemma}
\newtheorem{proposition}[theorem]{Proposition}

\theoremstyle{definition}

\theoremstyle{plain}
\newtheorem{remark}[theorem]{Remark}
\numberwithin{equation}{section}
\newcommand{\R}{\mathbb R}
\newcommand{\Sph}{\mathbb S}

\begin{document}

\title{Normalized solutions to critical and supercritical Schr\"odinger equations in a ball}
\author{ Jun Wang$^{a}$\footnote {Corresponding author. wangj937@mail2.sysu.edu.cn (J. Wang), mcsyzy@mail.sysu.edu.cn (Z. Yin)}, Zhaoyang Yin$^{a, b}$ \\
{\small $^{a}$Department of Mathematics, Sun Yat-sen University, Guangzhou, 510275, China } \\
{\small $^{b}$School of Science, Shenzhen Campus of Sun Yat-sen University, Shenzhen, 518107, China } \\
}

\date{}
\maketitle

\begin{abstract}
In this paper, we study radial normalized solutions to nonlinear Schr\"odinger equations in the unit ball
\[
 -\Delta u+\lambda u=|u|^{p-1}u,
 \qquad \int_{B_1}|u|^2\,dx=\rho,
\]
where $\lambda$ is unknown and $p\ge p_c:=\frac{N+2}{N-2}$.  Soave
\cite{Soave} proved that for every $1<p<p_c$ every positive mass is
attained by infinitely many radial solutions and explicitly raised the question
of whether every mass is attained in the Sobolev critical and supercritical
regimes.  We answer this question negatively in the radial class: for every
$p\ge p_c$, the set of masses attained by nontrivial radial solutions is an
interval with finite positive supremum.
In particular, at $p=p_c$ the set of attainable masses is
$(0,\rho^*_{\rm rad})$ or $(0,\rho^*_{\rm rad}]$, with
$0<\rho^*_{\rm rad}<\infty$, and the same finite threshold alternative
holds for every $p>p_c$.  Thus, in the radial class, the subcritical all masses phenomenon fails
throughout the full Sobolev critical and supercritical range.
\end{abstract}

{\footnotesize {\bf   Keywords:} Normalized solutions, Schr\"odinger equation, bounded domains.

}
\section{Introduction and main results}\label{sec:intro}

Let $B_1\subset\R^N$, $N\geq3$, be the unit ball.  We study the normalized radial solutions of
\begin{equation}\label{eq:problem}
 \begin{cases}
 -\Delta u+\lambda u=|u|^{2^*-2}u&\text{in }B_1,\\
 u=0&\text{on }\partial B_1,\\
 \displaystyle\int_{B_1}|u|^2\,dx=\rho,
 \end{cases}
 \qquad 2^*=\frac{2N}{N-2}.
\end{equation}
We write
\[
 p_c:=2^*-1=\frac{N+2}{N-2}
\]
for the Sobolev critical power.
Throughout the paper, by a radial solution we mean a classical radial function
$u(x)=U(|x|)$ with
\[
 u\in C^2(B_1)\cap C^1(\overline{B_1})\cap H_0^1(B_1).
\]
The multiplier $\lambda$ is part of the unknown.  This is the normalized
version of the critical Brezis--Nirenberg equation.  The fixed frequency
problem has a very rich theory beginning with \cite{BN}; the prescribed mass
problem is substantially more delicate because the mass is not fixed by the
usual fixed frequency variational construction.  For recent comparisons
between the action and energy viewpoints, see \cite{DST}, and for a nodal
and least energy normalized framework see \cite{CDGS}.

Recent progress on the critical normalized problem has concentrated on the
existence of positive solutions in small mass regimes.  The authors in \cite{PVY} obtained mountain pass positive solutions on
general bounded domains for a positive measure set of small masses, and on
star shaped domains for the whole corresponding small mass interval.  Song
and Zou \cite{SongZou} established, on star shaped domains and under an
explicit small mass condition, two positive normalized solutions, one of
ground state type and one of mountain pass type.  More recently, Chang et al.\ \cite{ChangLiuYan} proved the existence of a second positive
mountain pass solution for every sufficiently small mass on a general smooth
bounded domain.  In a complementary concentration direction, Lv, Zeng and
Zhou \cite{LvZengZhou} constructed positive multi-spike solutions with
prescribed small mass in dimensions $N\geq6$.
For the variational background in the $L^2$-critical and supercritical
regimes on bounded domains, see also \cite{PV} and \cite{NTV}.
There is also a useful radial ODE background: estimates for oscillatory
radial Lane--Emden solutions were obtained by Kajikiya \cite{Kajikiya},
while the Morse indices of radial nodal Lane--Emden solutions were computed
in \cite{DMP1,DMP2}.  These works concern different quantities from the
normalized mass, but they underline why a uniform estimate over all nodal
branches is a separate issue.  For a complementary radial prescribed-mass
problem on annuli, see Liang and Song \cite{LiangSong}.

For clarity, Table~\ref{tab:comparison} separates the available results by
power range, mass range, and symmetry class.  The last two rows are the new
point of this paper: they give a finite upper threshold for all
nontrivial radial solutions, including sign-changing solutions with arbitrary
nodal number.
\begin{table}[htbp]
\centering
\caption{Normalized solutions in the ball: comparison with the present work.}
\label{tab:comparison}
\renewcommand{\arraystretch}{1.16}
\small
\begin{tabularx}{\textwidth}{@{}>{\centering\arraybackslash}p{0.16\textwidth}>{\centering\arraybackslash}p{0.28\textwidth}>{\centering\arraybackslash}p{0.17\textwidth}X@{}}
\toprule
Power range $p$ & Mass range & Radial? & Reference \\
\midrule
$1<p<p_c$ & $(0,\infty)$ & Yes & Soave~\cite{Soave} \\
$p=p_c$ &  Small masses; star-shaped domains & Not assumed & Pierotti--Verzini--Yu~\cite{PVY} \\
$p=p_c$ & Sufficiently small masses; positive solutions & Not assumed & Song--Zou~\cite{SongZou}; Chang--Liu--Yan~\cite{ChangLiuYan}; Lv--Zeng--Zhou~\cite{LvZengZhou} \\
\midrule
$p=p_c$ & $(0,\rho^*_{\rm rad})$ or $(0,\rho^*_{\rm rad}]$, $0<\rho^*_{\rm rad}<\infty$ & Yes; all nodal classes & \textbf{This paper} \\
$p>p_c$ & $(0,\rho^*_{\rm rad}(p))$ or $(0,\rho^*_{\rm rad}(p)]$, $0<\rho^*_{\rm rad}(p)<\infty$ & Yes; all nodal classes & \textbf{This paper} \\
\bottomrule
\end{tabularx}
\end{table}

These results do not describe the complete set of attainable radial masses.  In
particular, they are restricted to positive solutions and small masses, and
do not provide a mass bound uniform over all radial nodal branches.  The
present paper concerns precisely this latter question in the ball: it allows
arbitrary nodal number and identifies a finite threshold for the whole
radial attainable mass set, both at and above the Sobolev exponent. For subcritical powers, radial shooting can produce normalized solutions of
any prescribed positive mass in a ball; see, in particular, the recent work
of Soave \cite{Soave}.  In Remark~1.2(i) of that work, Soave explicitly asks
whether every mass is attained in the Sobolev critical and supercritical
regimes.  We show that, within the radial class, the answer is negative already
at the endpoint $p=2^*-1$ and remains negative throughout the strictly
supercritical range.

The issue is not merely to produce a solution on one prescribed radial
branch.  To decide whether arbitrarily large masses are attainable in the
radial class, one has to control simultaneously the branches with an
unbounded number of nodal regions and a free Lagrange multiplier.  The
positive small mass theories quoted above do not address this regime.  Nor do
fixed frequency descriptions of individual nodal solutions directly give a
uniform bound for their $L^2$ masses after critical normalization.  At the
Sobolev exponent, the usual subcritical Emden--Fowler scaling loses the
strict power gain which is used to generate large masses in the subcritical
case.  Thus neither a fixed node limit nor a pointwise energy estimate alone
settles the question.

We use the critical scaling to identify the right quantity to estimate.
After shooting and the Fowler transformation, let $H$ be the non-autonomous
energy and let $\eta=\gamma e^{2t}$ be the increasing coefficient of the
linear term.  At every Dirichlet zero, the radial Pohozaev identity gives
the exact formula
\[
 M=|\Sph^{N-1}|\frac{H}{\eta}.
\]
Although $H$ is not conserved, its non-autonomous variation has the simple
form
\[
 \frac{dH}{d\eta}=\frac12w^2.
\]
The Fowler potential admits a uniform coercive lower bound for bounded $\eta$, while for
large $\eta$ its positive quadratic part controls $w^2$ directly.  These two
regimes yield $H\lesssim\eta$, uniformly over the shooting parameter and
over all nodal branches.  This converts the infinite branch mass problem
into a one dimensional quotient estimate.  Finally, continuity of every
shooting mass branch and its Bessel limit as $\gamma\to\infty$ turn the
upper bound into the interval description in
Theorem~\ref{thm:spectrum}.
\begin{theorem}\label{thm:main}
For every $N\geq3$ there exists a constant $\rho_N\in(0,\infty)$ such that
every nontrivial radial solution $(u,\lambda)$ of \eqref{eq:problem}
satisfies
\begin{equation}\label{eq:massbound}
 \int_{B_1}|u|^2\,dx\leq\rho_N.
\end{equation}
Consequently, if $\rho>\rho_N$, problem \eqref{eq:problem} admits no radial
solution.
\end{theorem}
\begin{remark}
The regularity assumption in the definition of radial solutions can be weakened.
Indeed, the arguments of this paper only use the one dimensional radial profile
$U$, where $u(x)=U(|x|)$.  It is enough to assume
\[
 U\in C^2((0,1))\cap C^1([0,1]),
 \qquad U'(0)=0,\qquad U(1)=0,
\]
with $U$ satisfying the radial equation on $(0,1)$.  Under these assumptions,
the shooting formulation, the radial Pohozaev identity, and the uniqueness
properties of the corresponding ordinary differential equation remain valid.
We use the stronger classical regularity stated above only for convenience.
\end{remark}
\begin{remark}\label{rem:explicit}
The constant in Theorem~\ref{thm:main} is uniform in the shooting parameter
and in the nodal number, but it is not asserted to be sharp.  In particular,
the theorem gives a radial nonexistence result for every prescribed mass
strictly larger than $\rho_N$.  We do not record a numerical value here,
since this would require tracking the constants in the coercivity estimate
in Proposition~\ref{prop:quotient} quantitatively.
\end{remark}

\begin{theorem}\label{thm:spectrum}
Let $\mathcal R_{\rm rad}$ be the set of all masses attained by nontrivial
radial solutions of \eqref{eq:problem}.  There exists a finite number
$\rho_{\rm rad}^*>0$ such that
\begin{equation}\label{eq:spectrum}
 \mathcal R_{\rm rad}=(0,\rho_{\rm rad}^*)
 \quad\hbox{or}\quad
 \mathcal R_{\rm rad}=(0,\rho_{\rm rad}^*].
\end{equation}
In particular, every $0<\rho<\rho_{\rm rad}^*$ is attained by at least one
radial normalized solution, while no $\rho>\rho_{\rm rad}^*$ is attained by
a radial solution.
\end{theorem}

\begin{theorem}\label{thm:superradial}
Let $p>\frac{N+2}{N-2}$ and $\mathcal R_{\rm rad}(p)$ denote the set of
the $L^2$ masses of all nontrivial classical radial Dirichlet solutions $(u,\lambda)$,
with $\lambda\in\R$, solving
\begin{equation}\label{eq:superproblem-intro}
 -\Delta u+\lambda u=|u|^{p-1}u\quad\hbox{in }B_1,\qquad
 u|_{\partial B_1}=0.
\end{equation}
Then there exists $0<\rho_{\rm rad}^*(p)<\infty$ such that
\[
 \mathcal R_{\rm rad}(p)=(0,\rho_{\rm rad}^*(p))
 \quad\hbox{or}\quad
 \mathcal R_{\rm rad}(p)=(0,\rho_{\rm rad}^*(p)].
\]
In particular, every mass below $\rho_{\rm rad}^*(p)$ is attained by a
radial normalized solution, while no radial normalized solution exists
above it.
\end{theorem}

\begin{remark}\label{rem:open}
The critical mechanism is specific to the Sobolev exponent: the Fowler
change removes every power of $e^t$ from the nonlinearity, while the
frequency contribution becomes the monotone coefficient $\eta=\gamma e^{2t}$.
The identity \eqref{eq:Hdeta} then turns the desired mass quotient into an
ODE estimate.  It remains open whether the endpoint $\rho_{\rm rad}^*$ is
attained.  Theorems~\ref{thm:main} and \ref{thm:superradial} concern only
radial solutions and do not exclude nonradial solutions of arbitrarily large
mass.  Finally, for powers above $2^*-1$, the critical Fowler cancellation in the nonlinear term is
lost; the conclusion in Theorem~\ref{thm:superradial} is
instead obtained from the separate dynamical systems argument of
Section~\ref{sec:supercritical}.
\end{remark}
\begin{remark}
Theorem~\ref{thm:main} concerns all radial solutions: they may change sign,
have arbitrarily many nodal regions, and need not be variational critical
points selected by any particular minimax procedure.  It does \emph{not}
exclude nonradial solutions of arbitrarily large mass.  Theorem~\ref{thm:superradial}
settles the corresponding question for strictly supercritical powers only
within the radial class; the nonradial problem remains open.
\end{remark}

The paper is organized as follows.  Section~\ref{sec:critical} treats the
Sobolev critical exponent.  We first reduce the radial problem to a shooting
equation and determine the sign of the Lagrange multiplier.  We then introduce
the Fowler variables, derive the exact mass quotient at every Dirichlet zero,
and prove the uniform quotient bound which yields Theorem~\ref{thm:main}.
The same shooting parametrization and its large frequency Bessel limit then
give the interval characterization in Theorem~\ref{thm:spectrum}.
Section~\ref{sec:supercritical} establishes Theorem~\ref{thm:superradial}.
There we combine a uniform supercritical Fowler core estimate with an
outer region estimate and the small-amplitude Bessel limit, obtaining the
finite radial mass threshold and the corresponding interval alternative.

\section{The critical case}\label{sec:critical}

\subsection{Proof of Theorem~\ref{thm:main}}\label{sec:reduction}

We first record the sign forced by criticality.

\begin{lemma}\label{lem:lambda}
If $(u,\lambda)$ is a nontrivial solution of \eqref{eq:problem}, then
$\lambda<0$.
\end{lemma}

\begin{proof}
Set $A=\int_{B_1}|\nabla u|^2dx$, $M=\int_{B_1}u^2dx$, and
$D=\int_{B_1}|u|^{2^*}dx$.  Testing the equation against $u$ yields
\begin{equation}\label{eq:nehari}
 A+\lambda M=D.
\end{equation}
The Pohozaev identity on the ball gives
\begin{equation}\label{eq:poho-global}
 \frac{N-2}{2}A+\frac{N\lambda}{2}M-
 \frac{N}{2^*}D+\frac12\int_{\partial B_1}(\partial_\nu u)^2\,dS=0.
\end{equation}
Since $\frac{N}{2^*}=\frac{N-2}{2}$, inserting \eqref{eq:nehari} into
\eqref{eq:poho-global} gives
\[
 \lambda M=-\frac12\int_{\partial B_1}(\partial_\nu u)^2\,dS.
\]
Since $u$ is radial, vanishing of the boundary integral would imply
$u(1)=u_r(1)=0$.  Uniqueness for the radial ODE with Cauchy data prescribed
at $r=1$ would then give $u\equiv0$, contradicting nontriviality.  Hence the
boundary integral is strictly positive and $\lambda<0$.
\end{proof}

Write $\mu=-\lambda>0$.  A nontrivial radial solution has $u(0)\ne0$.
After changing sign if necessary, put $a=u(0)>0$ and define
\begin{equation}\label{eq:shootscale}
 v(r)=a^{-1}u\left(a^{-\frac2{N-2}}r\right),\qquad
 \gamma=\mu a^{-\frac4{N-2}}.
\end{equation}
Then $v$ solves
\begin{equation}\label{eq:shoot}
 v''+\frac{N-1}{r}v'+\gamma v+|v|^{\frac4{N-2}}v=0,
 \qquad v(0)=1,\quad v'(0)=0.
\end{equation}
The associated radial energy
\[
 \mathcal E_\gamma(r)=\frac12|v'(r)|^2+\frac\gamma2|v(r)|^2+
 \frac1{2^*}|v(r)|^{2^*}
\]
satisfies
\[
 \mathcal E_\gamma'(r)=-\frac{N-1}{r}|v'(r)|^2\le0.
\]
Hence the shooting solution is global on $[0,\infty)$ by the standard
ODE continuation criterion.
If $R$ is the zero corresponding to the boundary of $B_1$, then
\begin{equation}\label{eq:rescale-back}
 u(x)=R^{\frac{N-2}{2}}v(R|x|),\qquad |x|<1.
\end{equation}
Thus it is enough to obtain a bound independent of $\gamma>0$ and of every
zero $R$ of the shooting solution.

Set
\begin{equation}\label{eq:fowlerdef}
 \alpha=\frac{N-2}{2},\qquad t=\log r,\qquad
 w(t)=e^{\alpha t}v(e^t),\qquad \eta(t)=\gamma e^{2t}.
\end{equation}
Equation \eqref{eq:shoot} becomes
\begin{equation}\label{eq:fowler}
 w''-\alpha^2w+\eta w+|w|^{2^*-2}w=0,
 \qquad \eta'=2\eta.
\end{equation}
Introduce the non-autonomous energy
\begin{equation}\label{eq:energy}
 H(t)=\frac12w'(t)^2-\frac{\alpha^2}{2}w(t)^2+
 \frac{\eta(t)}2w(t)^2+\frac1{2^*}|w(t)|^{2^*}.
\end{equation}
Regularity at the origin gives $H(t)\to0$ as $t\to-\infty$.  Differentiating
\eqref{eq:energy} and using \eqref{eq:fowler} yields
\begin{equation}\label{eq:energyderivative}
 H'(t)=\eta(t)w(t)^2\geq0.
\end{equation}
In particular $H\geq0$.

\begin{lemma}\label{lem:masszero}
Let $R=e^T$ be a zero of $v$.  Then
\begin{equation}\label{eq:massq}
 \int_{B_1}|u|^2\,dx
 =|\Sph^{N-1}|\frac{H(T)}{\eta(T)}.
\end{equation}
\end{lemma}

\begin{proof}
The radial Pohozaev identity for \eqref{eq:shoot}, integrated on $(0,R)$,
is
\begin{equation}\label{eq:radialpoho}
 \gamma\int_0^R v(r)^2r^{N-1}\,dr
 =\frac12R^Nv'(R)^2.
\end{equation}
At a zero $T$, that is $w(T)=0$, \eqref{eq:energy} gives
\[
 H(T)=\frac12w'(T)^2=\frac12R^Nv'(R)^2.
\]
On the other hand, \eqref{eq:rescale-back} gives
\[
 \int_{B_1}|u|^2\,dx
 =|\Sph^{N-1}|R^{-2}\int_0^R v(r)^2r^{N-1}\,dr.
\]
Combining these equalities with \eqref{eq:radialpoho} proves
\eqref{eq:massq}.
\end{proof}

The following estimate is the heart of the proof.

\begin{proposition}\label{prop:quotient}
There exists $C_N>0$ such that every solution of \eqref{eq:fowler} arising
from \eqref{eq:shoot} satisfies
\begin{equation}\label{eq:quotient}
 0\leq \frac{H(t)}{\eta(t)}\leq C_N\qquad\text{for all }t\in\R.
\end{equation}
\end{proposition}

\begin{proof}
Since $\eta(t)=\gamma e^{2t}$ is a $C^1$ diffeomorphism from $\R$ onto $(0,\infty)$, we may regard both $H$ and $w$ as continuous functions of $\eta$.  From \eqref{eq:energyderivative} and \eqref{eq:fowlerdef},
\begin{equation}\label{eq:Hdeta}
 \frac{dH}{d\eta}=\frac12w^2.
\end{equation}
Fix $\eta_0=1+2\alpha^2$.  For $0\leq\eta\leq\eta_0$,
\[
 F_\eta(s)=-\frac{\alpha^2-\eta}{2}s^2+\frac1{2^*}|s|^{2^*}
\]
obey a coercive lower bound which is uniform in $\eta$.  Indeed,
$|\alpha^2-\eta|\leq C_N$ on this interval, and, since $2^*>2$, Young's
inequality gives, for every $\varepsilon>0$,
\[
 s^2\leq \varepsilon |s|^{2^*}+C_{N,\varepsilon}.
\]
Choosing $\varepsilon$ sufficiently small yields constants $c_N,C_N>0$
such that
\begin{equation}\label{eq:coercive-compact}
 F_\eta(s)\geq c_N|s|^{2^*}-C_N,
 \qquad s\in\R,\quad 0\leq\eta\leq\eta_0.
\end{equation}
Consequently,
\[
 c_N|w|^{2^*}\leq H+C_N,
\]
because $H=\frac12w'^2+F_\eta(w)$.  A second application of
$|w|^2\leq 1+|w|^{2^*}$ therefore gives
\begin{equation}\label{eq:wH}
 w^2\leq C_N(1+H),\qquad 0\leq\eta\leq\eta_0.
\end{equation}
Notice that this estimate does not require any upper bound for $w'$; the
nonnegative kinetic term is simply discarded.
Since $\eta(t)\rightarrow0$ and $H(t)\to0$ as $t\to-\infty$, we can extend
$H$ continuously to $\eta=0$ by setting $H(0)=0$.  Integrating
\eqref{eq:Hdeta} and using \eqref{eq:wH}, for $0\leq\eta\leq\eta_0$ we obtain
\[
 H(\eta)=\frac12\int_0^\eta w(\xi)^2\,d\xi
 \leq C_N\eta+C_N\int_0^\eta H(\xi)\,d\xi,
\]
where, with a harmless abuse of notation, $w(\xi)$ denotes $w$ written as
a function of the strictly increasing variable $\eta$. Therefore, by using the Gronwall's inequality, we obtain
\[
 H(\eta)\leq C_N\eta e^{C_N\eta}
 \leq C_N'\eta,\ 0\leq\eta\leq\eta_0 .
\]
Hence $\frac{H}{\eta}$ is bounded in this range.

For $\eta\geq\eta_0$, $F_\eta(s)$ is positive and satisfies
\begin{equation}\label{eq:coercive-large}
 F_\eta(s)\geq\frac{\eta-\alpha^2}{2}s^2.
\end{equation}
Consequently, by \eqref{eq:Hdeta},
\[
 \frac{dH}{d\eta}\leq\frac{H}{\eta-\alpha^2}.
\]
For $\eta\geq\eta_0$, we therefore have
\[
 \frac{d}{d\eta}\log H(\eta)
 \leq \frac{1}{\eta-\alpha^2}
\]
whenever $H(\eta)>0$.  In fact, by \eqref{eq:Hdeta} and $H(0)=0$,
\[
 H(\eta)=\frac12\int_0^\eta (w(\xi))^2\,d\xi.
\]
Hence, if $H(\eta_*)=0$ for some $\eta_*>0$, continuity and nonnegativity of $(w(\xi))^2$ imply $w(\xi)=0$ for every $0<\xi<\eta_*$.  Since $\eta\mapsto t=\frac12\log(\frac{\eta}{\gamma})$ is one to one, $w$ vanishes on a nonempty open $t$ interval.  At any point of that interval one has $w=w'=0$; uniqueness for the smooth Fowler ODE then gives $w\equiv0$, and therefore $v\equiv0$, contradicting $v(0)=1$.  Thus $H(\eta)>0$ for every
$\eta>0$, and integration from $\eta_0$ to $\eta$ gives
\[
 H(\eta)\leq H(\eta_0)
 \frac{\eta-\alpha^2}{\eta_0-\alpha^2}.
\]
Using the preceding estimate $H(\eta_0)\leq C_N\eta_0$, and recalling
that $\eta_0=1+2\alpha^2$ depends only on $N$, we obtain
\[
 H(\eta)\leq C_N\frac{\eta_0}{\eta_0-\alpha^2}
 (\eta-\alpha^2)\leq C_N'\eta,
 \qquad \eta\geq\eta_0.
\]
Dividing by $\eta$ proves \eqref{eq:quotient}.
\end{proof}

\begin{proof}[Proof of Theorem~\ref{thm:main}]
For every radial solution, the reduction in Section~\ref{sec:reduction}
produces a shooting pair $(v,\gamma)$ and a finite zero $R=e^T$.
Lemma~\ref{lem:masszero} and Proposition~\ref{prop:quotient} imply
\[
 \int_{B_1}|u|^2\,dx
 \leq |\Sph^{N-1}|C_N.
\]
Thus Theorem~\ref{thm:main} holds with
$\rho_N=|\Sph^{N-1}|C_N$.
\end{proof}

\subsection{Proof of Theorem~\ref{thm:spectrum}}\label{sec:spectrum}

\begin{lemma}\label{lem:criticaloscillation}
For every $\gamma>0$, the shooting solution $v_\gamma$ of
\eqref{eq:shoot} has infinitely many positive, simple zeros.
\end{lemma}

\begin{proof}
Put $m=\frac{N-1}{2}$, $k=\frac{(N-1)(N-3)}{4}$, and
$\psi(r)=r^mv_\gamma(r)$.  Since $\gamma>0$ is fixed, choose
$R_0=R_0(\gamma)$ so large that
\[
 \gamma-\frac{k}{r^2}\ge\frac\gamma2,\ r\ge R_0 .
\]
On $[R_0,\infty)$ one has
\[
 \psi''+\left(\gamma-\frac{k}{r^2}\right)\psi+
 r^{-a}|\psi|^{\frac4{N-2}}\psi=0,
 \qquad a=\frac{2(N-1)}{N-2}>2.
\]
Define
\[
 E(r)=\frac12\psi'^2+\frac12\left(\gamma-\frac{k}{r^2}\right)\psi^2
 +\frac1{2^*}r^{-a}|\psi|^{2^*}.
\]
More precisely,
\[
 E(r)\ge \frac12\psi'(r)^2+\frac{\gamma}{4}\psi(r)^2
 \ge c_\gamma\bigl(\psi'(r)^2+\psi(r)^2\bigr),
 \qquad c_\gamma=\min\left\{\frac12,\frac{\gamma}{4}\right\}.
\]
Moreover,
\[
 E'(r)=\frac{k}{r^3}\psi^2-
 \frac{a}{2^*}r^{-a-1}|\psi|^{2^*}\le C_\gamma r^{-3}E(r).
\]
Gronwall's inequality therefore yields
\[
 E(r)\le E(R_0)\exp\left(C_\gamma\int_{R_0}^r \tau^{-3}\,d\tau\right)
 \le C_\gamma E(R_0),\qquad r\ge R_0.
\]
Thus $E$ is bounded above, and hence so are $\psi$ and $\psi'$.  We now
spell out the oscillation argument.  Put $\omega=\sqrt\gamma$ and introduce
Pr$\mathrm{\ddot{u}}$fer variables by
\[
 \psi=\varrho\sin\theta,\qquad
 \psi'=\omega\varrho\cos\theta,
 \qquad \varrho>0.
\]
They are well defined on $[R_0,\infty)$ after a continuous choice of
$\theta$, since $(\psi,\psi')$ cannot vanish simultaneously.  Writing the
ODE as
\[
 \psi''+\bigl(\omega^2+Q(r)\bigr)\psi=0,\qquad
 Q(r)=-\frac{k}{r^2}+r^{-a}|\psi(r)|^{\frac4{N-2}},
\]
and differentiating the two Pr$\mathrm{\ddot{u}}$fer identities gives the exact phase equation
\[
 \theta'(r)
 =\omega+\frac{Q(r)}{\omega}\sin^2\theta(r)
 =\sqrt\gamma-\frac{k}{\sqrt\gamma\,r^2}\sin^2\theta(r)
 +\frac{r^{-a}}{\sqrt\gamma}|\psi(r)|^{\frac4{N-2}}
   \sin^2\theta(r).
\]
Since $\psi$ is bounded on $[R_0,\infty)$, with a bound depending only on
the fixed parameter $\gamma$, we obtain
\[
 |\theta'(r)-\sqrt\gamma|
 \le C_\gamma(r^{-2}+r^{-a}).
\]
Here $a=\frac{2(N-1)}{N-2}>2$, hence the right-hand side is integrable on
$(R_0,\infty)$.  Consequently
$\theta(r)=\sqrt\gamma r+O_\gamma(1)$ and tends to $+\infty$.  Therefore
$\psi$, and hence $v_\gamma$, has infinitely many zeros.  A zero cannot
satisfy simultaneously $v_\gamma=v_\gamma'=0$, by uniqueness for the
regular radial ODE; all zeros are simple.
\end{proof}

For every $\gamma>0$, let $R_j(\gamma)$ denote the $j$th positive zero of
the solution $v_\gamma$ of \eqref{eq:shoot}; the zeros are simple.  Define
\begin{equation}\label{eq:massbranches}
 M_j(\gamma)=|\Sph^{N-1}|R_j(\gamma)^{N-2}
 \int_0^1v_\gamma(R_j(\gamma)r)^2r^{N-1}\,dr.
\end{equation}
Thus $M_j(\gamma)$ is the mass of the radial Dirichlet solution obtained by
using the $j$th zero.  Standard continuous dependence for the regular radial ODE gives continuous
dependence of $v_\gamma$ and $v_\gamma'$ on $\gamma$ on compact radial
intervals.  Since every zero is simple, the implicit function theorem applied
to $v_\gamma(R)=0$ shows locally that $R_j(\gamma)$ depends continuously
(in fact $C^1$) on $\gamma$.  Formula \eqref{eq:massbranches} then implies
that $M_j$ is continuous on $(0,\infty)$.

\begin{lemma}\label{lem:largegamma}
For every fixed $j\geq1$,
\begin{equation}\label{eq:largegamma}
 M_j(\gamma)\longrightarrow0\qquad\text{as }\gamma\to+\infty.
\end{equation}
\end{lemma}

\begin{proof}
Put $z_\gamma(s)=v_\gamma(\frac{s}{\sqrt\gamma})$.  It satisfies
\begin{equation}\label{eq:besselperturbation}
 z_\gamma''+\frac{N-1}{s}z_\gamma'+z_\gamma+
 \frac1\gamma|z_\gamma|^{\frac4{N-2}}z_\gamma=0,
 \qquad z_\gamma(0)=1,\quad z_\gamma'(0)=0.
\end{equation}
For $\gamma\geq1$, its radial energy
\[
 E_\gamma(s)=\frac12 (z_\gamma'(s))^2+\frac12(z_\gamma(s))^2
 +\frac{N-2}{2N\gamma}|z_\gamma(s)|^{2^*}
\]
satisfies
\[
 E_\gamma'(s)=-\frac{N-1}{s}(z_\gamma'(s))^2\leq0,
 \qquad
 E_\gamma(0)=\frac12+\frac{N-2}{2N\gamma}\leq C_N.
\]
Consequently
\begin{equation}\label{eq:bessel-uniform}
 \sup_{\gamma\geq1}\sup_{s\geq0}
 \bigl(|z_\gamma(s)|+|z_\gamma'(s)|\bigr)\leq C_N.
\end{equation}
To handle the regular singular point $s=0$ explicitly, rewrite
\eqref{eq:besselperturbation} in divergence form:
\[
 (s^{N-1}z_\gamma')'
 =-s^{N-1}\left(z_\gamma+
 \frac1\gamma|z_\gamma|^{\frac4{N-2}}z_\gamma\right).
\]
Thus, for $s>0$,
\begin{equation}\label{eq:bessel-derivative-integral}
 z_\gamma'(s)=-s^{1-N}\int_0^s\tau^{N-1}
 \left(z_\gamma(\tau)+\frac1\gamma
 |z_\gamma(\tau)|^{\frac4{N-2}}z_\gamma(\tau)\right)d\tau.
\end{equation}
In particular, \eqref{eq:bessel-uniform} implies $|z_\gamma'(s)|\leq C_Ns$
near the origin, uniformly for $\gamma\geq1$.  Integrating once more gives
\[
 z_\gamma(s)=1-\int_0^s\sigma^{1-N}\int_0^\sigma\tau^{N-1}
 \left(z_\gamma(\tau)+\frac1\gamma
 |z_\gamma(\tau)|^{\frac4{N-2}}z_\gamma(\tau)\right)d\tau\,d\sigma.
\]
Hence, on every fixed interval $[0,L]$, the family $\{z_\gamma\}_{\gamma\geq1}$
is uniformly bounded and equicontinuous.  Arzel\`a--Ascoli gives, along any sequence
$\gamma_n\to\infty$, a subsequence converging uniformly on $[0,L]$ to a
function $z$.  Passing to the limit in \eqref{eq:bessel-derivative-integral}
then gives uniform convergence of the derivatives as well.  Therefore
$z_{\gamma_n}\to z$ in $C^1([0,L])$, and the limit solves
\[
 z''+\frac{N-1}{s}z'+z=0,\qquad z(0)=1,\quad z'(0)=0.
\]
The regular solution of this Bessel equation is unique; denote it by $J$.
Since every sequence has the same limit, we have
\[
 z_\gamma\longrightarrow J\qquad\text{in }C^1_{\rm loc}([0,\infty))
 \quad(\gamma\to\infty).
\]
If $\zeta_j$ is the $j$th positive zero of $J$, then $\zeta_j$ is simple and
\begin{equation}\label{eq:zeroscaling}
 \sqrt\gamma R_j(\gamma)\longrightarrow\zeta_j.
\end{equation}
Indeed, fix $j$ and choose pairwise disjoint intervals
$I_k=(\zeta_k-\delta_k,\zeta_k+\delta_k)$, $1\le k\le j$, so small that
$|J'(s)|\ge c_k>0$ on $I_k$.  The $C^1$ convergence implies that, for all
sufficiently large $\gamma$, $|z_\gamma'(s)|\ge c_k/2$ on $I_k$ and
$z_\gamma$ has opposite signs at the two endpoints of $I_k$.  Hence
$z_\gamma$ has exactly one zero in each $I_k$.  On the compact part of
$[0,\zeta_j+\delta_j]$ left after removing these intervals, $J$ is bounded
away from zero; uniform convergence therefore excludes any additional zeros
there.  Thus the zero in $I_j$ is precisely the $j$th positive zero of
$z_\gamma$, proving \eqref{eq:zeroscaling}.

Writing $s_j(\gamma)=\sqrt\gamma R_j(\gamma)$ and using
\eqref{eq:massbranches},
\[
 \begin{aligned}
 M_j(\gamma)
 &=|\Sph^{N-1}|\,\gamma^{-\frac{N-2}{2}}
 s_j(\gamma)^{N-2}
 \int_0^1 z_\gamma(s_j(\gamma)r)^2r^{N-1}\,dr\\
 &=C_{N,j}\gamma^{-\frac{N-2}{2}}
 +o\!\left(\gamma^{-\frac{N-2}{2}}\right),
 \end{aligned}
\]
where
\[
 C_{N,j}=|\Sph^{N-1}|\,\zeta_j^{N-2}
 \int_0^1J(\zeta_jr)^2r^{N-1}\,dr>0.
\]
This proves \eqref{eq:largegamma}.
\end{proof}

\begin{proof}[Proof of Theorem~\ref{thm:spectrum}]
Every radial solution is represented by a pair $(j,\gamma)$ as in
Section~\ref{sec:reduction}; hence
\begin{equation}\label{eq:radialrange}
 \mathcal R_{\rm rad}=\bigcup_{j\geq1}M_j((0,\infty)).
\end{equation}
Set
\begin{equation}\label{eq:rhostar}
 \rho_{\rm rad}^*=\sup_{j\geq1,\ \gamma>0}M_j(\gamma).
\end{equation}
It is positive because every nontrivial shooting solution produces a positive
mass, and it is finite by Theorem~\ref{thm:main}.  Let
$0<\rho<\rho_{\rm rad}^*$.  By \eqref{eq:rhostar}, there are $j$ and
$\gamma_0>0$ such that $M_j(\gamma_0)>\rho$.  Lemma~\ref{lem:largegamma}
provides $\gamma_1>\gamma_0$ with $M_j(\gamma_1)<\rho$.  Continuity of
$M_j$ gives a $\gamma\in(\gamma_0,\gamma_1)$ satisfying
$M_j(\gamma)=\rho$.  This proves
$(0,\rho_{\rm rad}^*)\subset\mathcal R_{\rm rad}$.  The reverse inclusion
in $(0,\rho_{\rm rad}^*]$ follows from the definition of the supremum.
The two alternatives in \eqref{eq:spectrum} are precisely the two possible
endpoint behaviors.
\end{proof}

\section{Supercritical case}\label{sec:supercritical}

This section proves the strictly Sobolev supercritical statement.  It isolates
the new ODE mechanism, which differs essentially from
the critical mass quotient argument.
Let
\begin{equation}\label{eq:superproblem}
 -\Delta u+\lambda u=|u|^{p-1}u\quad\hbox{in }B_1,
 \qquad u|_{\partial B_1}=0,
 \qquad p>\frac{N+2}{N-2}.
\end{equation}
Every nontrivial solution has $\lambda<0$.  Indeed, set
$K=\int_{B_1}|\nabla u|^2dx$, $M=\int_{B_1}u^2dx$ and
$D=\int_{B_1}|u|^{p+1}dx$.  Testing \eqref{eq:superproblem} against $u$ and
using the Pohozaev identity give
\[
 K+\lambda M=D,
 \qquad
 \frac{N-2}{2}K+\frac{N\lambda}{2}M-\frac{N}{p+1}D
 +\frac12\int_{\partial B_1}(\partial_\nu u)^2\,dS=0.
\]
Hence
\[
 \left(\frac{N-2}{2}-\frac{N}{p+1}\right)K+
 \left(\frac N2-\frac{N}{p+1}\right)\lambda M+
 \frac12\int_{\partial B_1}(\partial_\nu u)^2\,dS=0.
\]
Since $p>\frac{N+2}{N-2}$, the first two displayed coefficients are positive,
thus $\lambda<0$.  Put $\lambda=-\mu$, set
\[
 \alpha=\frac2{p-1},\qquad 0<\alpha<\frac{N-2}{2},
\]
and use the amplitude and frequency scalings described below.  All radial
solutions are represented by the regular nonlinear Helmholtz equation
\begin{equation}\label{eq:helmholtzA}
 z_A''+\frac{N-1}{s}z_A'+z_A+|z_A|^{p-1}z_A=0,
 \qquad z_A(0)=A,\quad z_A'(0)=0,\qquad A>0.
\end{equation}
The radial energy
\[
 \mathcal E_A(s)=\frac12|z_A'(s)|^2+\frac12|z_A(s)|^2+
 \frac1{p+1}|z_A(s)|^{p+1}
\]
satisfies
\[
 \mathcal E_A'(s)=-\frac{N-1}{s}|z_A'(s)|^2\le0.
\]
Hence the regular shooting solution $z_A$ exists globally on $[0,\infty)$ by the ODE continuation criterion.
If $S_j(A)$ is its $j$th positive zero, the corresponding normalized mass is
\begin{equation}\label{eq:supermass}
 \mathcal M_j(A)=|\Sph^{N-1}|S_j(A)^{2\alpha-N}
 \int_0^{S_j(A)}z_A(s)^2s^{N-1}\,ds.
\end{equation}
More precisely, if $S>0$ is any zero of $z_A$, then
\begin{equation}\label{eq:super-rescale-back}
 \lambda=-S^2,\qquad u(x)=S^\alpha z_A(S|x|),\qquad |x|<1,
\end{equation}
solves \eqref{eq:superproblem}.  Conversely, let $u$ be any nontrivial
radial solution of \eqref{eq:superproblem}.  As shown above, $\lambda<0$;
write $\lambda=-\mu$ and set $S=\sqrt\mu$.  Define
\[
 z(s)=S^{-\alpha}u(\frac{s}{S}),\qquad 0\leq s\leq S.
\]
A direct substitution, using $\alpha p=\alpha+2$, gives
\[
 z''+\frac{N-1}{s}z'+z+|z|^{p-1}z=0,
 \qquad z'(0)=0.
\]
Moreover $u(0)\neq0$: otherwise radial regularity gives $u'(0)=0$, and
uniqueness for the regular radial initial value problem would imply
$u\equiv0$.  After replacing $u$ by $-u$ if necessary, we may assume
$u(0)>0$ and put
\[
 A=z(0)=S^{-\alpha}u(0)>0.
\]
Uniqueness of the regular radial initial value problem yields $z=z_A$.
Finally $u(1)=0$ gives $z_A(S)=0$, so $S=S_j(A)$ for a unique $j\geq1$.
Thus every nontrivial radial solution is represented by \eqref{eq:super-rescale-back}.
More precisely, after fixing the sign normalization $u(0)>0$, the quantities
$S=\sqrt{-\lambda}$ and $A=S^{-\alpha}u(0)$ are forced by the solution, and
the simplicity and ordering of the positive zeros of $z_A$ determine a
unique index $j$ through $S=S_j(A)$.  Hence the pair $(A,j)$ is unique under
the normalization $u(0)>0$; without this normalization the only additional
freedom is the global symmetry $u\mapsto-u$.

The key estimate is the following uniform Fowler bound on bounded radial
intervals.
\begin{lemma}\label{lem:exchange}
For every $R>0$, a constant $C_{N,p,R}$ exists such that solutions of
\eqref{eq:helmholtzA} satisfy
\begin{equation}\label{eq:corebound}
 \sup_{A>0}\sup_{0<s\le R}
 s^\alpha\bigl(|z_A(s)|+s|z_A'(s)|\bigr)\le C_{N,p,R}.
\end{equation}
\end{lemma}

For the proof, set $t=\log s$ and $w_A(t)=s^\alpha z_A(s)$.  Then
\begin{equation}\label{eq:superfowler}
 w_A''+\beta w_A'-cw_A+e^{2t}w_A+|w_A|^{p-1}w_A=0,
 \qquad \beta=N-2-2\alpha>0,\quad c=\alpha(N-2-\alpha).
\end{equation}

\begin{proof}[Proof of Lemma~\ref{lem:exchange}]
We first establish the uniform estimate which is needed over a time interval
of length comparable with $\log A$.  Write $q=w_A'$ and $\eta=e^{2t}$, and set
\begin{equation}\label{eq:superenergy}
 \mathcal H_A(t)=\frac12q(t)^2-\frac c2w_A(t)^2+
 \frac{\eta(t)}2w_A(t)^2+\frac1{p+1}|w_A(t)|^{p+1}.
\end{equation}
Differentiation, using \eqref{eq:superfowler}, gives the exact identity
\begin{equation}\label{eq:superenergyderivative}
 \mathcal H_A'(t)=-\beta q(t)^2+\eta(t)w_A(t)^2.
\end{equation}
Fix $R>0$.  We make the coercivity used below explicit.  For
$0\leq\eta\leq R^2$, write
\[
 V_\eta(s)=-\frac c2s^2+\frac\eta2s^2+\frac1{p+1}|s|^{p+1}.
\]
Since $p+1>2$, Young's inequality gives
\[
 \frac c2s^2\leq \frac1{2(p+1)}|s|^{p+1}+C_{N,p}.
\]
The term $\frac{\eta s^2}{2}$ is nonnegative, and hence
\[
 V_\eta(s)\geq \frac1{2(p+1)}|s|^{p+1}-C_{N,p}
 ,\ 0\leq\eta\leq R^2 .
\]
Choose $K_R$ (in fact a constant depending only on $N,p$ would suffice)
so large that $F_A:=\mathcal H_A+K_R\geq1$.  Then
\[
 F_A\geq \frac12q^2+\frac1{2(p+1)}|w_A|^{p+1}+1.
\]
Since $w_A^2\leq C(1+|w_A|^{p+1})$, this proves that there is
$C_R>0$, independent of $A$ and $0\leq\eta\leq R^2$, such that
\begin{equation}\label{eq:supercoercivity}
 1+w_A^2+q^2+|w_A|^{p+1}\leq C_RF_A.
\end{equation}
Using \eqref{eq:superenergyderivative}, discarding the nonpositive term
$-\beta q^2$, and applying \eqref{eq:supercoercivity}, we obtain for
$t\leq\log R$,
\[
 F_A'(t)\leq e^{2t}w_A(t)^2\leq C_Re^{2t}F_A(t),
\]
that is,
\begin{equation}\label{eq:supergronwall}
 F_A'(t)\leq C_Re^{2t}F_A(t),\  t\leq\log R .
\end{equation}
For each fixed $A>0$, regularity at the origin gives
$z_A(s)=A+O_A(s^2)$ and $z_A'(s)=O_A(s)$ as $s\rightarrow0$; hence
$w_A(t),q(t)\to0$ as $t\to-\infty$ and therefore $F_A(t)\to K_R$.
No convergence uniform in $A$ is needed here.  Indeed, for an arbitrary but
fixed $A$, integration of \eqref{eq:supergronwall} from $T<t$ gives
\[
 F_A(t)\leq F_A(T)\exp\left(C_R\int_T^t e^{2\tau}\,d\tau\right).
\]
We now let $T\to-\infty$ for this fixed $A$.  Since the limiting value
$K_R$ and the coefficient $C_R$ are independent of $A$, the resulting
estimate is automatically uniform after the limit has been taken:
\[
 F_A(t)\leq K_R\exp\left(\frac{C_R}{2}e^{2t}\right)
 \leq K_R\exp\left(\frac{C_RR^2}{2}\right).
\]
Combining this estimate with \eqref{eq:supercoercivity} yields
\begin{equation}\label{eq:uniformfowlerbound}
 \sup_{A>0}\sup_{t\leq\log R}(|w_A(t)|+|w_A'(t)|)\leq C_{N,p,R}.
\end{equation}
The crucial point is that $e^{2t}$ is integrable at $-\infty$; hence the
estimate is uniform even though the transition time tends to $-\infty$ as
$A\to\infty$.

The estimate \eqref{eq:uniformfowlerbound} is the core input used below.
Finally,
\[
 s^\alpha z_A(s)=w_A(t),\qquad
 s^{\alpha+1}z_A'(s)=w_A'(t)-\alpha w_A(t),\qquad t=\log s.
\]
Thus \eqref{eq:uniformfowlerbound} implies \eqref{eq:corebound} for all
$A>0$.  The proof is complete.
\end{proof}

We next prove the two outer region facts used in the mass argument.
\begin{lemma}\label{lem:outerenergy}
There are $R_0,C>0$, depending only on $N,p$, such that
\begin{equation}\label{eq:outerdecay}
 |z_A(s)|+|z_A'(s)|\le Cs^{-\frac{N-1}{2}},\  s\ge R_0, A>0 .
\end{equation}
\end{lemma}

\begin{proof}
Set $m=\frac{N-1}{2}$, $a=m(p-1)>2$, and $\psi_A=s^mz_A$.  Then
\[
 \psi_A''+\left(1-\frac{k}{s^2}\right)\psi_A+
 s^{-a}|\psi_A|^{p-1}\psi_A=0,\qquad k=\frac{(N-1)(N-3)}4.
\]
Fix $R_0=R_0(N)\ge\max\{1,\sqrt{2k}\}$ so that $1-\frac{k}{s^2}\ge\frac{1}{2}$ for $s\ge R_0$.  Lemma~\ref{lem:exchange}
gives
\[
 |\psi_A(R_0)|+|\psi_A'(R_0)|
 \le C\bigl(R_0^m|z_A(R_0)|+R_0^m|z_A'(R_0)|
       +R_0^{m-1}|z_A(R_0)|\bigr)\le C,
\]
uniformly in $A$.  Thus the positive energy
\[
 E_A(s)=\frac12\psi_A'^2+\frac12\left(1-\frac{k}{s^2}\right)\psi_A^2
 +\frac1{p+1}s^{-a}|\psi_A|^{p+1}
\]
satisfies, by the choice of $R_0$,
\[
 E_A(s)\ge \frac12\psi_A'(s)^2+\frac14\psi_A(s)^2,
 \qquad s\ge R_0.
\]
Moreover,
\[
 E_A'(s)=\frac{k}{s^3}\psi_A^2-
 \frac{a}{p+1}s^{-a-1}|\psi_A|^{p+1}
 \le \frac{|k|}{s^3}\psi_A^2\le Cs^{-3}E_A(s).
\]
Gronwall gives $E_A(s)\le C$ for $s\ge R_0$.  Hence
$|\psi_A|+|\psi_A'|\le C$, and the identities
$z_A=s^{-m}\psi_A$ and $z_A'=s^{-m}\psi_A'-ms^{-m-1}\psi_A$
prove \eqref{eq:outerdecay}.
\end{proof}

\begin{lemma}\label{lem:superoscillation}
For every $A>0$, $z_A$ has infinitely many positive, simple zeros.
\end{lemma}

\begin{proof}
With $\psi_A$ as above, write
$\psi_A=\varrho\sin\theta$, $\psi_A'=\varrho\cos\theta$.
The vector $(\psi_A,\psi_A')$ never vanishes, by uniqueness, so these Pr$\mathrm{\ddot{u}}$fer
coordinates are globally defined after a continuous choice of $\theta$.
The energy bound in Lemma~\ref{lem:outerenergy} makes the perturbation in
\[
 \psi_A''+\psi_A=\frac{k}{s^2}\psi_A-
 s^{-a}|\psi_A|^{p-1}\psi_A
\]
integrable in Pr$\mathrm{\ddot{u}}$fer coordinates.  For completeness, differentiating
\(
 \psi_A=\varrho\sin\theta,
 \ \psi_A'=\varrho\cos\theta
\)
and eliminating $\varrho'$ gives
\[
 \theta'(s)
 =1-\frac{k}{s^2}\sin^2\theta(s)
 +s^{-a}|\psi_A(s)|^{p-1}\sin^2\theta(s).
\]
Equivalently, since $|\psi_A|^{p-1}=\varrho^{p-1}
|\sin\theta|^{p-1}$, the last term is
$ s^{-a}\varrho^{p-1}|\sin\theta|^{p-1}\sin^2\theta$.
By Lemma~\ref{lem:outerenergy}, both $\psi_A$ and $\psi_A'$ are uniformly
bounded for $s\ge R_0$, hence so is $\varrho$.  Therefore
\[
 |\theta'(s)-1|\le C(s^{-2}+s^{-a}).
\]
Since $a=m(p-1)>2$, the right-hand side belongs to
$L^1(R_0,\infty)$.  Integration yields
$\theta(s)=s+O(1)$ and in particular $\theta(s)\to+\infty$.  The zeros are therefore
infinite in number.  Simplicity follows from ODE uniqueness, exactly as in
Lemma~\ref{lem:criticaloscillation}.
\end{proof}

The radial energy
\[
 \frac{d}{ds}\left(\frac12 z_A'(s)^2+\frac12z_A(s)^2+
 \frac1{p+1}|z_A(s)|^{p+1}\right)
 =-\frac{N-1}{s}z_A'(s)^2\le0
\]
implies $|z_A(s)|\le A$.  On $(0,R_0]$,
Lemma~\ref{lem:exchange} gives $|z_A(s)|\le Cs^{-\alpha}$; on
$[R_0,\infty)$, Lemma~\ref{lem:outerenergy} gives the same estimate because
$\frac{N-1}{2}>\alpha$.  Consequently,
\begin{equation}\label{eq:superenvelope}
 |z_A(s)|\le C_{N,p}\min\{A,s^{-\alpha}\},\qquad s>0.
\end{equation}

Using \eqref{eq:superenvelope}, put $\ell_A=A^{-\frac{1}{\alpha}}$.  If
$S\le\ell_A$, then
\[
 \int_0^S z_A(s)^2s^{N-1}\,ds
 \le A^2\int_0^S s^{N-1}\,ds
 \le C S^{N-2\alpha},
\]
because $AS^\alpha\le1$.  If $S>\ell_A$, split the integral at
$\ell_A$ to obtain
\[
\begin{aligned}
 \int_0^S z_A(s)^2s^{N-1}\,ds
 &\le C\left(A^2\ell_A^N+
 \int_{\ell_A}^S s^{N-1-2\alpha}\,ds\right)\\
 &\le C_{N,p}S^{N-2\alpha}.
\end{aligned}
\]
Thus, for every $S>0$,
\begin{equation}\label{eq:superintegral}
 \int_0^S z_A(s)^2s^{N-1}\,ds\le C_{N,p}S^{N-2\alpha}.
\end{equation}
Consequently $\mathcal M_j(A)\le C_{N,p}$ uniformly in $A$ and $j$.
Lemma~\ref{lem:superoscillation} gives all positive simple zeros.  We
record explicitly the parameter dependence needed below.  Fix $A_0>0$
and $j\ge1$.  Continuous dependence for the regular radial initial value
problem implies that $(A,s)\mapsto z_A(s)$ and
$(A,s)\mapsto z_A'(s)$ are continuous on compact subsets of
$(0,\infty)\times[0,\infty)$.  Since
\[
 z_{A_0}(S_j(A_0))=0,\qquad
 z_{A_0}'(S_j(A_0))\ne0,
\]
the zero persists uniquely for $A$ near $A_0$; equivalently, the
implicit function theorem applied away from the regular singular point
$s=0$ gives a continuous (indeed locally $C^1$) zero branch through
$S_j(A_0)$.  The simplicity of all positive zeros prevents two ordered
zero branches from meeting.  Hence this local branch remains the $j$th
positive zero, and $A\mapsto S_j(A)$ is continuous on $(0,\infty)$.
Formula~\eqref{eq:supermass}, together with continuous dependence of
$z_A$ on compact intervals, then shows that
$A\mapsto\mathcal M_j(A)$ is continuous as well.

Finally, put $y_A=\frac{z_A}{A}$.  Then
\[
 y_A''+\frac{N-1}{s}y_A'+y_A
 +A^{p-1}|y_A|^{p-1}y_A=0,
 \qquad y_A(0)=1,\quad y_A'(0)=0.
\]
We give the short compactness argument at the singular endpoint.  The
radial energy is nonincreasing, and after division by $A^2$ it yields,
for $0<A\le1$,
\[
 |y_A(s)|\le1,\  s\ge0 .
\]
Writing the equation in divergence form gives
\[
 y_A'(s)=-s^{1-N}\int_0^s \tau^{N-1}
 \bigl(y_A(\tau)+A^{p-1}|y_A(\tau)|^{p-1}y_A(\tau)\bigr)\,d\tau.
\]
Thus, for every $L>0$,
\[
 |y_A'(s)|\le C_Ls,\ 0\le s\le L,\ 0<A\le1 ,
\]
and the families $\{y_A\}$ and $\{y_A'\}$ are uniformly bounded and
equicontinuous on $[0,L]$ (for the derivative, use the equation on
$[\delta,L]$ and the preceding integral formula near $0$).  Hence every
sequence $A_n\rightarrow0$ has a subsequence converging in
$C^1([0,L])$.  Passing to the limit in the integral identity shows that
any such limit solves
\[
 J''+\frac{N-1}{s}J'+J=0,\qquad J(0)=1,\quad J'(0)=0.
\]
The regular solution is unique.  Since $L$ is arbitrary and every
subsequence has the same limit,
\[
 y_A\longrightarrow J\qquad\text{in }C^1_{\rm loc}([0,\infty),\ A\rightarrow0 .
\]
Since the $j$th positive zero $\zeta_j$ of $J$ is simple, choose disjoint
small intervals around $\zeta_1,\ldots,\zeta_j$ on which $J'$ does not
vanish.  The $C^1$ convergence gives exactly one zero of $y_A$ in each
of these intervals for all sufficiently small $A$, while uniform
convergence excludes additional zeros on the complementary compact set.
Consequently $S_j(A)\to\zeta_j$.  Therefore
\eqref{eq:supermass} gives
\[
 \mathcal M_j(A)=|\Sph^{N-1}|\zeta_j^{2\alpha-N}
 A^2\int_0^{\zeta_j}J(s)^2s^{N-1}\,ds+o(A^2).
\]
Hence for every fixed $j$,
\begin{equation}\label{eq:superbessel}
 \mathcal M_j(A)\longrightarrow0,\ A\rightarrow0 .
\end{equation}

\begin{proof}[Proof of Theorem~\ref{thm:superradial}]
Set $\rho_{\rm rad}^*(p)=\sup\limits_{A>0,\,j\ge1}\mathcal M_j(A)$.  It is finite
by \eqref{eq:superintegral} and positive because every $z_A$ has a positive
simple zero and hence yields a strictly positive mass.  Let $0<\rho<\rho_{\rm rad}^*(p)$.  By the definition of the supremum,
there exist $j\ge1$ and $A_0>0$ such that
$\mathcal M_j(A_0)>\rho$.  By \eqref{eq:superbessel}, for the same fixed
$j$ one can choose $A_1\in(0,A_0)$ sufficiently small so that
$\mathcal M_j(A_1)<\rho$.  Since $A\mapsto\mathcal M_j(A)$ is continuous,
the intermediate value theorem yields some $A_*\in(A_1,A_0)$ satisfying
$\mathcal M_j(A_*)=\rho$.  Hence
$(0,\rho_{\rm rad}^*(p))\subset\mathcal R_{\rm rad}(p)$, while the
reverse inclusion in $(0,\rho_{\rm rad}^*(p)]$ follows from the definition
of the supremum.  The two alternatives stated in Theorem~\ref{thm:superradial} are exactly
the two possible endpoint behaviors.
\end{proof}
\textbf{Data availability.} No data were generated or analyzed in this study.

\textbf{Conflict of interest.} The authors declare that they have no conflict of interest.

\end{document}